\documentclass[11pt,reqno]{amsart}
\usepackage{textcomp}
\usepackage{graphics}
\usepackage{amssymb, a4wide}
\usepackage{amsmath}
\usepackage{graphicx}
\usepackage{epsfig}
\DeclareMathOperator{\arcsinh}{arcsinh}

\newtheorem{theorem}{Theorem}
\newtheorem{example}{Example}
\newtheorem{remark}{Remark}

\newtheorem{corollary}{Corollary}
\newtheorem{definition}{Definition}

\title{Bour's Theorem of Spacelike Surfaces in Minkowski 4--Space}

\author[M. Babaarslan]{Murat Babaarslan}
\address{Yozgat Bozok University, Department of Mathematics, 66100, Yozgat, Turkey}
\email{murat.babaarslan@bozok.edu.tr}

\author[B. Bekta\c{s} Dem\.{i}rc\.i]{Burcu Bekta\c s Dem\.{i}rc\.i}
\address{Fatih Sultan Mehmet Vak{\i}f University, Hal\.{I}\c{c} Campus, Faculty of Engineering,
Department of Civil Engineering, 34445, Beyo\u{g}lu, İstanbul, Turkey(Corresponding author)}
\email{bbektas@fsm.edu.tr}

\author[Y. K{\"u}\c{c}{\"u}kar{\i}kan]{Yas\.in K{\"u}\c{c}{\"u}kar{\i}kan}
\address{Yozgat Bozok University, School of Graduates Studies, Department of Mathematics, 66100, Yozgat, Turkey}
\email{kucukarikanyasin@gmail.com}

\begin{document}
\maketitle

\begin{abstract}
In this paper, we study on three kinds of spacelike helicoidal surfaces in Minkowski $4$--space. First, we give an isometry between such helicoidal surfaces and rotational surfaces which is a kind of generalization of Bour theorem in Minkowski $3$--space to Minkowski $4$--space. Then, we investigate geometric properties for such isometric surfaces having same Gauss map. By using these results, we give the parametrizations of isometric pair of surfaces. As a particular case, we examine the right helicoidal surfaces in view of Bour's theorem. Also, we present some examples by choosing the components of the profile curves and the parameters of the surfaces via Mathematica. 
\end{abstract}

\section{Introduction}
In daily life, examples of surfaces can be seen everywhere. 
For instance, balloons, innertubes, cans and soap films provide physical models of surfaces (see \cite{Oprea}). 
Rotational surfaces are one of the most well--known surfaces
and the use of rotational surfaces is fundamental in physics and engineering. 
Since they are found in many places in nature, they are also used in geometric modeling. Cans, glasses and furniture legs are some examples of rotational surfaces (see \cite{Karacan}).

A natural generalization of the rotational surfaces are helicoidal surfaces. 
Also, we often come across examples of helicoidal surfaces in nature, science and engineering such as creeper plants, helicoidal stairs, helicoidal conveyors, parking garage ramps, helicoidal paths, helical channels, helicoidal towers and skyscrapers (see \cite{Babaarslan1}).

Many mathematicians have obtained significant results regarding helicoidal surfaces in different ambient spaces 
(see \cite{Carmo, Baikoussis, Beneki, Ji1, Choi1, Ji2, Choi2, Ji3, Kim, Lopez, Babaarslan1, Babaarslan2, Babaarslan3}). 
One of the most notable result is a generalization of the relationship between a right helicoid and a catenoid.
In this context, E. Bour gave the following statement about an isometry between a helicoidal surface 
and a rotational surface in $\mathbb{E}^3$ which is called Bour's theorem. \\
\textbf{Bour's theorem.}
A generalized helicoid is isometric to a rotational surface so that helices on the helicoid correspond to parallel circles on the rotational surface \cite{Bour}. 

On the other hand, the right helicoid and catenoid are both members of a one-parameter family of isometric minimal surfaces and they have the same Gauss map. In \cite{Ikawa1}, 
T. Ikawa discussed these properties for isometric surfaces according to Bour's theorem in $\mathbb{E}^3$. 
Also, E. G{\"u}ler and Y. Yayl{\i} studied Bour's theorem 
for generalized helicoidal surface in $\mathbb{E}^3$ and they investigated the minimality condition and having same Gauss map for isometric surfaces in \cite{Guler1}.
Moreover, E. G{\"u}ler et. al. gave the relation between the Laplace--Beltrami operator and the curvatures of helicoidal surfaces in Euclidean $3$--space and they also discussed Bour's theorem on the Gauss map of the helicoidal surfaces in \cite{Guler2}.

After these works, it seems natural question 
whether Bour's theorem 
have a counterpart in Minkowski space. 
For this purpose, T. Ikawa \cite{Ikawa2} studied Bour's theorem of spacelike and timelike helicoidal surfaces in Minkowski 3--space with non--null and null axis and he determined the parametrizations the pair of isometric surfaces which have same Gauss map in $\mathbb{E}^3_1$. 
Then, E. G{\"u}ler and A. T. Vanl{\i} examined Bour's theorem on generalized helicoid with null axis in Minkowski 3-space in \cite{Guler3}.

In 2017, D. T. Hieu and N. N. Thang \cite{Hieu} studied Bour's theorem for helicoidal surfaces of Euclidean $4-$space and they proved that
if the Gauss maps of isometric surfaces are same, then they are hyperplanar and minimal.

The notion of Minkowski 4--space or Minkowski space--time is strongly related to Einstein’s special theory of relativity and and it was proposed as a model for spacetime in the special theory of relativity by Minkowski in 1907, (see \cite{Ratcliffe}). The Minkowski 4--space has more complicated and richer geometric structures than the Euclidean 4--space
due to the difference between their metrics, angles and the motions of particles (see \cite{Babaarslan3}).

In Minkowski $4$--space, there are three types rotation with $2$--dimensional axis, called elliptic, hyperbolic and parabolic rotations which leave the spacelike, timelike and degenerate planes invariant. By using these rotations and the definition of helicoidal surface, M. Babaarslan and N. S{\"o}nmez gave the parametrizations of three types of helicoidal surfaces in Minkowski $4$--space, \cite{Babaarslan4}.

In this paper, we examine Bour's theorem on three types of spacelike helicoidal surfaces in Minkowski 4--space, that is, Bour's theorem in Minkowski 3--space will be generalized to Minkowski 4--space. 
After that, we determine the parametrizations of the pair of isometric spacelike surfaces having same Gauss map and we give their geometric properties such as hyperplanar and minimal. Also, we illustrate some examples by choosing the components of the profile curves and the parameters of the surfaces.

\section{Preliminaries}
Let a metric $\Tilde{g}$ be defined by
\begin{equation}
\label{eq1}
\Tilde{g}(x,y)=x_{1}y_{1}+x_{2}y_{2}+x_{3}y_{3}-x_{4}y_{4}
\end{equation}
for the vectors  $x=(x_{1},x_{2},x_{3},x_{4})$ and $y=(y_{1},y_{2},y_{3},y_{4})$ in Euclidean $4-$space $\mathbb{E}^4$. With the metric $\Tilde{g}$, the space $\mathbb{E}^4$ is called as Minkowski $4-$space and it is denoted by $\mathbb{E}^4_{1}$.
A vector $v$ in $\mathbb{E}^4_1$ is called spacelike if 
$\langle v,v\rangle>0$ or $v=0$, timelike if $\langle v,v\rangle<0$
and lightlike if $\langle v,v\rangle=0$ and $v\neq 0$. 

Assume that ${X}:D\subset\mathbb{R}^2\rightarrow\mathbb{E}^4_{1}$ 
is a smooth parametric surface in $\mathbb{E}^4_{1}$ with a coordinate system $\{u,v\}$, where $D$ is an open subset of $\mathbb{R}^2$. 
Then, the tangent plane of the surface $X$ at the point $p={X(u,v)}$ is generated by $X_{u}$ and $X_{v}$.

The coefficients of first fundamental form of $X$ is given by
\begin{equation}
\label{eq2}
g_{11}=\langle X_{u}, X_{u}\rangle,
\text{ }g_{12}=g_{21}=\langle X_{u}, X_{v}\rangle, 
\text{ }g_{22}=\langle X_{v}, X_{v}\rangle. 
\end{equation}
Thus, the induced metric $g$ on $X$ is
\begin{equation}
g=g_{11} d{u}^{2}+ 2g_{12}d{u}d{v}+g_{22}dv^2.
\end{equation}
If ${W}=\mbox{det}(g)=g_{11}g_{22}-
g_{12}^2\not=0 $, then the surface $X$ is non--degenerate.
Otherwise, it is degenerate one. 
Throughout the article, we are not interested in degenerate case.
If ${W}>0$, then the surface $X$ in $\mathbb{E}^4_1$ is spacelike and if ${W}<0$, then it is a timelike surface.  

Let $\{e_{1},{e}_{2},{N}_{1},{N}_{2}\}$ be 
a local orthonormal frame field on $X$ in $\mathbb{E}^4_1$
such that 
$e_{1},{e}_{2}$ are tangent to $X$ and $N_{1},N_{2}$ are normal to $X$. 
Then, the coefficients of the second fundamental form of $X$ according to $N_{i}$, $(i=1,2)$ are given by 
 \begin{equation}
 \label{eq3}
 b_{11}^{i}=\langle X_{uu},  N_{i}\rangle,
 \;\;b_{12}^{i}=b_{21}^{i}=\langle X_{uv}, N_{i}\rangle, 
 \;\;b_{22}^{i}=\langle X_{vv}, N_{i}\rangle.  
 \end{equation}
Thus, the mean curvature vector $H$ of $X$ in $\mathbb{E}^4_1$ is given by
\begin{equation}
\label{eq4}
H={\epsilon}_{1}H_{1}N_{1}+\mathbf{\epsilon}_{2}H_{2}N_{2},
\end{equation}
where the components $H_i$ of $H$ is
$\displaystyle{H_{i}=
\frac{b_{11}^{i}g_{22}-2b_{12}^{i}g_{12}+b_{22}^{i}g_{11}}{2{W}}}$ for $i=1,2$ and the Gauss curvature $K$ of $X$ in $\mathbb{E}^4_1$ is
\begin{equation}
\label{eq6}
K=\frac{\epsilon_{1}(b_{11}^1b_{22}^1-(b_{12}^1)^2)+
{\epsilon_{2}(b_{11}^2b_{22}^2-(b_{12}^2})^2)}{{W}}
\end{equation}
for ${\epsilon}_{1}=\langle N_{1},N_{1}\rangle$, 
${\epsilon}_{2}=\langle N_{2},N_{2}\rangle$.
If the mean curvature vector $H$ of $X$ is zero, 
then $X$ is said to be a minimal (maximal) surface in $\mathbb{E}^4_1$ and if the Gauss curvature of $X$ is zero, 
then it is called as developable (flat) surface in $\mathbb{E}^{4}_{1}$. 
Also, $X$ is said to be a marginally trapped surface if the mean curvature vector $H$ is lightlike.
 
Let $G(2,4)$ denote the Grassmanian manifold consisting of 
all oriented $2-$planes through the origin in $\mathbb{E}^4_1$
and $\bigwedge^2 \mathbb{E}^{4}_{1}$ the vector space obtained 
by the exterior product of $2-$vectors in $\mathbb{E}^4_1$.
Let $f_{i_1}\wedge f_{i_2}$ and $g_{i_1}\wedge g_{i_2}$ be two
vectors in $\bigwedge^2 \mathbb{E}^{4}_{1}$, where 
$\{f_1,f_2,f_3,f_4\}$ and $\{g_1,g_2,g_3,g_4\}$ are two orthonormal
bases of $\mathbb{E}^4_1$. Define an indefinite inner product
$\langle \cdot,\cdot \rangle$ on $\bigwedge^2 \mathbb{E}^{4}_{1}$ by 
\begin{equation}
 \langle f_{i_1}\wedge f_{i_2}, g_{i_1}\wedge g_{i_2}\rangle 
 =\mbox{det}(\langle f_{i_l}, g_{j_k}\rangle ).
\end{equation}
Therefore, we identify $\bigwedge^2 \mathbb{E}^{4}_{1}$ with 
some pseudo--Euclidean space $\mathbb{E}^6_t$ for some positive integer $t$. Then, $G(2,4)$ is canonically embedded in 
$\bigwedge^2 \mathbb{E}^{4}_{1}\cong\mathbb{E}^6_t$. 
For a surface $X$ in $\mathbb{E}^4_1$, 
the Gauss map $\nu:X\rightarrow G(2,4)\subset\mathbb{E}_{t}^{6}$ 
is a smooth map 
which carries a point $p\in X$ into the oriented $2$--plane through the origin of $\mathbb{E}^4_1$ obtained from the parallel translation of the tangent space of $X$ at $p$ in $\mathbb{E}^4_1$.
Since $\{e_1, e_2\}$ is an orthonormal tangent frame on $X$, then 
the Gauss map $\nu$ is given by
\begin{equation}
 \nu(p)=(e_1\wedge e_2)(p).   
\end{equation} 
We refer \cite{ChenP} to see the details. 
Moreover, the Gauss map is also defined by using the normal space of $X$ in $\mathbb{E}^4_1$, \cite{Chen1}. 
From now on, we assume that $X$ is a spacelike surface in $\mathbb{E}^4_1$, that is, $W>0$. Then, we can choose an orthonormal tangent frame field $e_{1},{e}_{2}$ on $X$ as follows
 \begin{equation}
 \label{eq7}
     e_{1}=\frac{1}{\sqrt{g_{11}}}X_{u},
     \;\;\; e_{2}=\frac{1}{\sqrt{Wg_{11}}}
     (g_{11}X_{v}-g_{12}X_{u}).
 \end{equation}
According to the chosen frame field, the Gauss map $\nu$ of 
$X$ can be written as 
 \begin{equation}
 \label{eq8}
     \nu=\frac{1}{\sqrt{W}}X_{u} \wedge X_{v}.
 \end{equation}
The definition of helicoidal surfaces in $\mathbb{E}^4_{1}$ can be given as follows.
\begin{definition}
\cite{Babaarslan4}
Let $\beta:I\subset\mathbb{R}\longrightarrow \Pi\subset\mathbb{E}^4_1$ be a smooth curve 
in a hyperplane $\Pi\subset\mathbb{E}^4_1$, 
${P}$ be a $2-$plane in $\Pi$ and $\ell$ be a line parallel to ${P}$. A helicoidal surface in $\mathbb{E}^4_1$ is defined as a rotation of the curve $\beta$ around ${P}$ with a translation along the line $\ell$. Here, the speed of translation is proportional to the speed of this rotation. 
\end{definition} 
Also, the three types of helicoidal surfaces in $\mathbb{E}^4_{1}$ were defined in \cite{Babaarslan4} given as follows.

Let $\{\eta_{1},\eta_{2},\eta_{3},\eta_{4}\}$ be a standard orthonormal basis for $\mathbb{E}^4_1$, where $\eta_{1}=(1,0,0,0)$,
$\eta_{2}=(0,1,0,0)$, $\eta_{3}=(0,0,1,0)$ and $\eta_{4}=(0,0,0,1)$.
Then, we choose a timelike $2-$plane ${P}_1$ generated by 
$\{\eta_{3},\eta_{4}\}$, ${\Pi}_1$ a hyperplane generated by 
$\{\eta_{1},\eta_{3},\eta_{4}\}$ and a line $l_{1}$ generated by $\eta_{4}$.
Assume that 
$\beta_1:I\longrightarrow\Pi_1\subset\mathbb{E}^4_1,
\;\beta_{1}(u)=\left(x(u),0,z(u),w(u)\right)$
is a smooth regular curve lying in $\Pi_1$ defined on an open interval $I\subset\mathbb{R}$ with $x(u)\neq 0$.
For $0\leq v<2\pi$ and $\lambda\in \mathbb{R^{+}}$, 
we consider the surface $X_1$
\begin{equation}
\label{eq9}
   X_{1}(u,v) =(x(u)\cos v,x(u)\sin v,z(u),w(u)+\lambda v),
\end{equation}
which is the parametrization of the helicoidal surface obtained the rotation of the curve $\beta_1$ 
that leaves the timelike plane ${P}_1$ pointwise fixed  
followed by the translation along $l_1$. 
The surface $X_1$ in $\mathbb{E}^4_1$
is called a helicoidal surface of type I.
If $w$ is a constant function, then $X_{1}$ is called as right helicoidal surface of type I. Also, if $z$ is a constant function, then $X_{1}$ is just a helicoidal surface in $\mathbb{E}^3_{1}$.

Let $\{\eta_{1},\eta_{2},\eta_{3},\eta_{4}\}$ be a standard orthonormal basis for $\mathbb{E}^4_1$.
We choose a spacelike $2-$plane $P_{2}$ generated by $\{\eta_{1},\eta_{2}\}$, a hyperplane $\Pi_{2}$ generated by $\{\eta_{1},\eta_{2},\eta_{4}\}$ and a line $l_{2}$ generated by $\{\eta_{1}\}$. 
Assume that 
$\beta_{2}:I\longrightarrow\Pi_{2}\subset\mathbb{E}^4_1,
\;\beta_{2}(u)=\left(x(u),y(u),0, w(u)\right)$
is a smooth regular curve lying in $\Pi_{2}$ 
defined on an open interval $I\subset\mathbb{R}$ 
$w(u)\neq 0.$ 
For $v\in\mathbb{R}$ and $\lambda\in \mathbb{R^{+}}$, 
we consider the surface $X_{2}$ 
\begin{equation}
\label{eq10}
X_{2}(u,v)=(x(u)+\lambda v,y(u),w(u)\sinh v,w(u)\cosh v)
\end{equation}
which is the parametrization of the helicoidal surface obtained the rotation of the curve $\beta_{2}$ which leaves 
the spacelike plane ${P}_{2}$ pointwise fixed followed by the translation along $l_{2}$. 
The surface $X_{2}$ in $\mathbb{E}^4_1$
is called a helicoidal surface of type II.
If $x$ is a constant function, then $X_{2}$ is called as right helicoidal surface of type II. Also, if $y$ is a constant function, then $X_{2}$ is just a helicoidal surface in $\mathbb{E}^3_{1}$.

Let us consider the pseudo--orthonormal basis $\{\eta_{1},\eta_{2},\mathbf{\xi}_{3},\mathbf{\xi}_{4}\}$ 
of $\mathbb{E}^4_{1}$ such that $\mathbf{\xi}_{3}=\frac{1}{\sqrt{2}}(\eta_{4}-\eta_{3})$ and $\mathbf{\xi}_{4}=\frac{1}{\sqrt{2}}(\eta_{3}+\eta_{4})$. 
Then, we choose a degenerate $2-$plane $P_{3}$ generated by $\{\eta_{1},\mathbf{\xi}_{3}\}$, a hyperplane $\Pi_{3}$ 
generated by $\{\eta_{1},\mathbf{\xi}_{3},\mathbf{\xi}_{4}\}$ and a line $l_{3}$ generated by $\{\mathbf{\xi}_{3}\}$. 
Then, the orthogonal transformation $T_{3}$ of $\mathbb{E}^4_{1}$ which leaves the degenerate plane $P_{3}$ invariant is given by $T_{3}(\eta_{1})=\eta_{1},\ T_{3}(\eta_{2})=\eta_{2}+\sqrt{2}v\mathbf{\xi}_{3},\ T_{3}(\mathbf{\xi}_{3})=\mathbf{\xi}_{3}$ and $T_{3}(\mathbf{\xi}_{4})=\sqrt{2}v\eta_{2}+v^{2}\mathbf{\xi}_{3}+\mathbf{\xi}_{4}$. 
If the profile curve $\beta_{3}$ is given by $\beta_{3}(u)=x(u)\eta_{1}+z(u)\mathbf{\xi}_{3}+w(u)\mathbf{\xi}_{4}$ in $\Pi_{3},$ where $u\in I\subset \mathbb{R}$ and $w(u)\neq 0$.
Then, we consider the surface $X_{3}$ 
\begin{equation}
\label{eq11}
X_{3}(u,v)=x(u)\eta_{1}+\sqrt{2}vw(u)\eta_{2}+(z(u)+v^{2}w(u)
+\lambda v)\mathbf{\xi}_{3}+w(u)\mathbf{\xi}_{4}
\end{equation}
which is the parametrization of 
the helicoidal surface obtained a rotation 
of the curve $\beta_{3}$ which leaves the degenerate plane 
${P}_{3}$ pointwise fixed
followed by the translation along $l_{3}$. 
The surface $X_{3}$ in $\mathbb{E}^4_1$
is called the helicoidal surface of type III. 
If $w$ is a constant function, then $M_{3}$ is called as right helicoidal surface of type III.
\begin{remark}
For $\lambda=0$, the three kinds of helicoidal surfaces in $\mathbb{E}^4_1$ given by \eqref{eq9}--\eqref{eq11} 
become the rotational surfaces of hyperbolic, elliptic and parabolic type in $\mathbb{E}^4_1$, respectively. 
For the details, see \cite{Dursun} and \cite{Bektas}.
Throughout this article, the corresponding rotational surfaces 
are denoted by $R_1, R_2$ and $R_3$. 
\end{remark}

\section{Bour's Theorem of Spacelike Surfaces}
In this section, we give Bour's theorem for three types of spacelike helicoidal surfaces in $\mathbb{E}^4_1$ and 
we study such isometric surfaces having same Gauss map. 
\subsection{Helicoidal Surface of Type I}
Assume that $X_1$ is a spacelike helicoidal surface of type I in $\mathbb{E}^4_1$ given by \eqref{eq9}. 
From a simple calculation, we have 
\begin{equation}
\label{type1eq2}
g_{11}= x'^2(u)+z'^2(u)-w'^2(u),\;\;g_{12}=g_{21}=-\lambda w'(u),\;\;
g_{22}=x^2(u)-\lambda^2 
\end{equation}
with ${W}=(x^2(u)-\lambda^2)(x'^2(u)+z'^2(u))-x^2(u)w'^2(u)>0$
for all $u\in I\subset\mathbb{R}$.
Then, we choose an orthonormal frame field 
$\{e_{1},e_{2}, N_{1}, N_{2}\}$ on $X_1$ in $\mathbb{E}^4_1$ 
such that $e_1, e_2$ are tangent to $X_1$ and $N_1, N_2$ are normal to $X_1$:
\begin{align}
\label{type1eq3}
\begin{split}
e_{1}&=\frac{1}{\sqrt{g_{11}}}X_{1u},\;\;\;\;
e_{2}=\frac{1}{\sqrt{Wg_{11}}}(g_{11}X_{1v}-{g_{12}X_{1u}}),\\
N_{1}&=\frac{1}{\sqrt{x'^2+z'^2}}(z'\cos v, z'\sin v, -x',0),\\
N_{2}=&\frac{1}{\sqrt{W(x'^2+z'^2)}}
(xx'w'\cos v-\lambda(x'^2+z'^2)\sin v,
xx'w'\sin v+\lambda(x'^2+z'^2)\cos v,\\
&xz'w', x(x'^2+z'^2))
\end{split}
\end{align}
with 
$\langle e_{1},e_{1}\rangle=\langle e_{2},e_{2}\rangle=\langle N_{1},N_{1}\rangle=-\langle N_{2},N_{2}\rangle=1$. 
By a direct computations, we get the coefficients of the second fundamental form as follows
\begin{align}
\label{type1eq4}
    \begin{split}
     &b_{11}^1=\frac{x'' z'- x' z''}{\sqrt{{x'^2}+z'^2}},\;\;\; 
     b_{12}^1=b_{21}^1=0,\;\;\;
     b_{22}^1=\frac{-x z'}{\sqrt{x'^2+z'^2}},\\
     &b_{11}^2=\frac{x(w'(x'x''+z'z'')-w''(x'^2+z'^2))}
    {\sqrt{W(x'^2+z'^2)}},\;\; 
     b_{12}^2=b_{21}^2=\frac{\lambda x'\sqrt{x'^2+z'^2}}{\sqrt{W}},\\
    &b_{22}^2=\frac{- x^2x'w'}{\sqrt{W(x'^2+z'^2)}}.
    \end{split}
\end{align}
Thus, we find the components $H_{i}^{X_{1}},\;(i=1,2)$ 
of the mean curvature
and from the equation \eqref{eq6}, we also compute the Gauss curvature $K^{X_{1}}$ of $X_1$ in $\mathbb{E}^4_1$, respectively, as follows
\begin{align}
\label{type1eq5} 
\begin{split}
&H_{1}^{X_{1}}=\frac{-xz'(x'^2+z'^2-w'^2)+(x^2-
\lambda^2)(x''z'-x'z'')}{2W\sqrt{x'^2+z'^2}},\\
&H_{2}^{X_{1}}=\frac{x'w'(2\lambda^2-x^2)(x'^2+z'^2)+x^2x'w'^3-x(x^2-\lambda^2)(x'(x'w''-x''w')+z'(z'w''-w'z''))}
{2{W}^{3/2}\sqrt{x'^2+z'^2}}
\end{split}
\end{align}
and 
\begin{equation}
\label{type1eq6}
K^{X_{1}}=\frac{xz'(x^2-\lambda^2)(x'z''-x''z')+x^3w'(x''w'-x'w'')+ \lambda^2 x'^2}{W^2}.
\end{equation}

The following theorem is a generalization of the classical Bour's theorem  for a spacelike helicoidal surfaces of type I in  $\mathbb{E}^4_{1}$.
\begin{theorem}
\label{type1thm1}
A spacelike helicoidal surface of type I in $\mathbb{E}^4_{1}$ given by \eqref{eq9} is isometric to a spacelike rotational surface \begin{equation}
\label{type1izR}
  R_{1}(u,v)=\begin{bmatrix}
\sqrt{x^2(u)-\lambda^2}\cos{\left(v-\int{\frac{\lambda w'(u)}{x^2(u)-\lambda^2}du}\right)}\\ \sqrt{x^2(u)-\lambda^2}\sin{\left(v-\int{\frac{\lambda w'(u)}{x^2(u)-\lambda^2}du}\right)}\\
\int{\frac{a(u)x(u)x'(u)}{\sqrt{x^2(u)-\lambda^2}}du}\\
\int{\frac{b(u)x(u)x'(u)}{\sqrt{x^2(u)-\lambda^2}}du}
\end{bmatrix}
\end{equation}
so that helices on the spacelike helicoidal surface of type I correspond to parallel circles on the spacelike rotational surface, where $a(u)$ and $b(u)$ are differentiable functions satisfying the following equation:
 \begin{equation}
 \label{type1eq7}
    a^2(u)-b^2(u)=\frac{x^2(u)(z'^2(u)-w'^2(u))-\lambda^2(x'^2(u)+z'^2(u))}{x^2(u)x'^2(u)}
 \end{equation}
with $x'(u)\not=0$ for all $u\in I\subset\mathbb{R}$.
\end{theorem}
 
\begin{proof}
Let $X_1$ be a spacelike helicoidal surface in $\mathbb{E}^4_1$ 
defined by \eqref{eq9}. From the equation \eqref{type1eq2}, 
we have the induced metric of $X_1$
as follows:
\begin{equation}
\label{type1mtrc1}
 ds^2 _{X_1}=(x'^2(u)+z'^2(u)-w'^2(u))du^2-2\lambda w'(u)dudv + (x^2(u)-\lambda^2)dv^2.
\end{equation}
Then, we find a curve on $X_1$ which is orthogonal to a helix on $X_1$, that is, this curve is $v$--parameter curve of $X_1$. 
From orthogonality condition, we get 
\begin{equation}
\label{type1eq8}
    -\lambda w'(u)du+(x^2(u)-\lambda^2)dv=0.
\end{equation}
By solving the equation \eqref{type1eq8}, we find 
\begin{equation}
    v=\int{\frac{\lambda w'(u)}{x^2(u)-\lambda^2}du}+c,
\end{equation}
where c is an arbitrary constant.
If we take $\bar{v}=c$, then we have 
\begin{equation}
\label{type1eq8-a}
\overline{v}=v-\int{\frac{\lambda w'(u)}{x^2(u)-\lambda^2}du}.
\end{equation}
From equation \eqref{type1eq8-a}, we can easily obtain the following equation:
\begin{equation}
\label{type1eq9}
    dv=d\overline{v}+\frac{\lambda w'(u)}{x^2(u)-\lambda^2}du.
\end{equation}
Substituting the equation \eqref{type1eq9} in \eqref{type1mtrc1}, 
the induced metric of $X_1$ becomes
\begin{equation}
\label{type1eq10}
 ds^2 _{X_1}=\left(x'^2(u)+z'^2(u)-w'^2(u)-\frac{\lambda^2 w'^2(u)}{x^2(u)-\lambda^2}\right)du^2 + (x^2(u)-\lambda^2)d\overline{v}^2.
\end{equation}
On the other hand, 
the spacelike rotational surface $R_1$ in $\mathbb{E}^4_{1}$
related to $X_1$ is given by 
\begin{equation}
\label{type1rtk}
    R_{1}(k,t)=(n(k)\cos{t},n(k)\sin{t},s(k),r(k)).
\end{equation}
We know that the induced metric of $R_1$ is given by
\begin{equation}
\label{type1eq11}
 ds^2 _{R_1}=(n'^2(k)+s'^2(k)-r'^2(k))dk^2 + n^2(k)dt^2
\end{equation}
with $n'^2(k)+s'^2(k)-r'^2(k)>0$. 
Comparing the equations \eqref{type1eq10} and \eqref{type1eq11}, 
we get an isometry between $X_1$ and $R_1$ by taking
$u=k, \bar{v}=t$, $n(k)=\sqrt{x^2(u)-\lambda^2}$ 
and 
\begin{equation}
\label{type1eq12}
 n'^2(k)+s'^2(k)-r'^2(k)= x'^2(u)+z'^2(u)-w'^2(u)-\frac{\lambda^2 w'^2(u)}{x^2(u)-\lambda^2}. 
\end{equation}
Also, we note that due to the fact that $X_1$ is a spacelike surface in $\mathbb{E}^4_1$, 
$x^2(u)-\lambda^2>0$ for $u\in I\subset\mathbb{R}$. 
Say $a(u)=\frac{s'(u)}{n'(u)}$ and 
$b(u)=\frac{r'(u)}{n'(u)}$. Then, we find 
\begin{equation}
  s(u)=\int{\frac{a(u)x(u)x'(u)}
  {\sqrt{x^2(u)-\lambda^2}}du}\;\;\mbox{and}\;\; r(u)=\int{\frac{b(u)x(u)x'(u)}
  {\sqrt{x^2(u)-\lambda^2}}du}.
\end{equation}
Thus, we get the spacelike rotational surface $R_1$ given by \eqref{type1izR} which is isometric to $X_1$. 
If we choose a helix as $X_1(u_0,v)$, 
where $u_{0}$ is an arbitrary constant, 
it corresponds to the parallel circle $R_{1}(u_0,v)=(\sqrt{x^2(u_0)-\lambda^2}\cos{v},\sqrt{x^2(u_0)-\lambda^2}\sin{v},0,0)$ lying on the $x_1x_2$--plane with the radius $\sqrt{x^2(u_0)-\lambda^2}$. 
\end{proof} 

For later use, we find the components of the mean curvature vector
of the spacelike rotational surface $R_1$ given by \eqref{type1izR} as follows:
\begin{align}
\label{type1izm1}
\begin{split}
H_1^{R_1}&=\frac{(\lambda^2-x^2)a'+axx'(b^2-a^2-1)}
{2xx'(1+a^2-b^2)\sqrt{(1+a^2)(x^2-\lambda^2)}},\\
H_2^{R_1}&=
\frac{(aa'b(x^2-\lambda^2)-(x^2-\lambda^2)b'+bxx'(b^2-1)
+a^2((\lambda^2-x^2)b'-bxx')}{2xx'(1+a^2-b^2)^{3/2}
\sqrt{(1+a^2)(x^2-\lambda^2)}}.
\end{split}
\end{align}
Then, we consider isometric surfaces according to Bour's theorem whose Gauss maps are same. 
\begin{theorem}
\label{type1thm2}
Let $X_{1}$ be a spacelike helicoidal surface of type I given by \eqref{eq9} and $R_{1}$ be a spacelike rotational surface  of type I given by \eqref{type1izR} in $\mathbb{E}^4_{1}$. If they have the same Gauss map, then they are hyperplanar and minimal.
\end{theorem}

\begin{proof}
Assume that $X_{1}$ be a spacelike helicoidal surface of type I in $\mathbb{E}^4_{1}$ given by \eqref{eq9} and $R_{1}$ be a spacelike rotational surface  of type I in $\mathbb{E}^4_{1}$  given by \eqref{type1izR}. Now, we define the Gauss maps 
of $X_{1}$ and $R_{1}$ as follows.\\
Let $\{\eta_{1},\eta_{2},\eta_{3},\eta_{4}\}$ be the standard 
orthonormal bases of $\mathbb{E}_1^4$.
Denote $\eta_{ij}=\eta_{i} \wedge \eta_{j}$ for ${i,j=1,2,3,4}$ and
$i<j$.
Then, using equation \eqref{eq8}, 
we get the Gauss maps of
$X_1$ and $R_1$, respectively, as follows
\begin{align}
\label{type1GH}
\nu_{X_{1}} =&\frac{1}{\sqrt{W}}(xx'\eta_{12}+xz'\sin{v}\eta_{13}+(\lambda x'\cos{v}+xw'\sin{v})\eta_{14}-xz'\cos{v}\eta_{23}\notag\\
&+(\lambda x'\sin{v}-xw'\cos{v})\eta_{24}+ \lambda z'\eta_{34})
\end{align}
and
\begin{align}
\label{type1GR}
\nu_{R_{1}}=&\frac{1}{\sqrt{W}}
\Bigg(xx'\eta_{12}+axx'\sin\left(v-\int\frac{\lambda w'(u)}{x^2(u)-\lambda^2}du\right)\eta_{13}
\notag\\&+bxx'\sin\left(v-\int\frac{\lambda w'(u)}{x^2(u)-\lambda^2}du\right)\eta_{14}-axx'\cos\left(v-\int\frac{\lambda w'(u)}{x^2(u)-\lambda^2}du\right)\eta_{23}\notag\\&-bxx'\cos\left(v-\int\frac{\lambda w'(u)}{x^2(u)-\lambda^2}du\right)\eta_{24}\Bigg).\end{align}
If $X_1$ and $R_1$ have the same Gauss map, 
then comparing \eqref{type1GH} and \eqref{type1GR}, we get the following system of equations
\begin{equation}
  \label{type1eq12}
    x(u)z'(u)\sin{v}=a(u)
    x(u)x'(u)\sin{\left(v-\int\frac{\lambda w'(u)}{x^2(u)-\lambda^2}du\right)},
\end{equation}
\begin{equation}
  \label{type1eq13}
    x(u)z'(u)\cos{v}=a(u)
    x(u)x'(u)\cos{\left(v-\int\frac{\lambda w'(u)}{x^2(u)-\lambda^2}du\right)},
\end{equation}
\begin{equation}
  \label{type1eq14}
    \lambda x'(u)\cos{v}+x(u)w'(u)\sin{v}=b(u)x(u)x'(u)\sin{\left(v-\int
    \frac{\lambda w'(u)}{x^2(u)-\lambda^2}du\right)},
\end{equation}
\begin{equation}
  \label{type1eq15}
    \lambda x'(u)\sin{v}-x(u)w'(u)\cos{v}=-b(u)x(u)x'(u)\cos{\left(v-\int
    \frac{\lambda w'(u)}{x^2(u)-\lambda^2}du\right)},
\end{equation}
\begin{equation}
  \label{type1eq16}
    \lambda z'(u)=0.
\end{equation}
Since $\lambda \neq 0 $, $z'(u)=0$ from equation \eqref{type1eq16}.
Then, the equations \eqref{type1eq12} and \eqref{type1eq13}
give $a(u)=0$. 
Therefore, it can be easily seen that 
the surfaces $X_{1}$ and $R_{1}$ are hyperplanar,
that is, they are lying in $\mathbb{E}^3_1$.
In this case, the equations \eqref{type1eq5} 
and \eqref{type1izm1} give the components of the mean curvatures 
of $X_1$ and $R_1$ as follows:
\begin{equation}
\label{type1eq17}
H_{1}^{X_{1}}=H_{1}^{R_{1}}=0
\end{equation}
and
\begin{align}
\label{type1eq18}
\begin{split}
H_{2}^{X_{1}}&=\frac{x'^2w'(2\lambda^2-x^2)
+x^2w'^3+x(x^2-\lambda^2)(x''w'-x'w'')}{2{W}^{3/2}},\\
H_{2}^{R_{1}}&=
\frac{-(x^2-\lambda^2)b'+bxx'(b^2-1)}{2(1-b^2)^{3/2}
xx'\sqrt{x^2-\lambda^2}}.
\end{split}
\end{align}
On the other hand, the equation \eqref{type1eq7} gives 
\begin{equation}
 \label{type1eq18a}   
 b^2(u)=\frac{x^2(u)w'^2(u)+\lambda^2 x'^2(u)}{x^2(u)x'^2(u)}. 
\end{equation}
Using the equation \eqref{type1eq18a} in \eqref{type1eq18}, we get
\begin{equation}
 \label{type1eq18b}
H_{2}^{R_{1}}=\frac{x^2w'(x'^2w'(2\lambda^2-x^2)+x^2w'^3
+x(x^2-\lambda^2)(x''w'-x'w''))}{2W\sqrt{(x^2w'^2+
\lambda^2x'^2)(x^2-\lambda^2)}}.    
\end{equation}
Thus, it can be seen that $H_2^{R_1}=x^2w'H_2^{X_1}$. 
Moreover, using equations \eqref{type1eq14} and \eqref{type1eq15}, we obtain the following equations
\begin{equation}
\label{type1eq21}
    x(u)w'(u)=b(u)x(u)x'(u)\cos{\left(\int\frac{\lambda w'(u)}{x^2(u)-\lambda^2}du\right)},
\end{equation}
\begin{equation}
\label{type1eq22}
    \lambda x'(u)=-b(u)x(u)x'(u)\sin{\left(\int\frac{\lambda w'(u)}{x^2(u)-\lambda^2}du\right)}.
\end{equation}
Considering the equations \eqref{type1eq21} and \eqref{type1eq22} together, we have
\begin{equation}
\label{type1eq23}
    \frac{x(u)w'(u)}
    {\lambda x'(u)}=-\cot{\left(\int\frac{\lambda w'(u)}{x^2(u)-\lambda^2}du\right)}.
\end{equation}
If we take the derivative of the equation \eqref{type1eq23} with respect to $u$, the equation \eqref{type1eq23} becomes
\begin{equation}
\label{type1eq24}
    \lambda^2(xx'w''+w'(2x'^2-
    xx''))+x^2(w'(w'^2 -x'^2)+x(x''w'-x'w''))=0
\end{equation}
which implies $H_{2}^{X_{1}}=H_{2}^{R_{1}}=0$. 
Thus, we get the desired results.
\end{proof}

\begin{theorem}
\label{type1thm2a}
Let $X_{1}$ be a spacelike helicoidal surface of type I given by \eqref{eq9} and $R_{1}$ be a spacelike rotational surface of type I given by \eqref{type1izR} in $\mathbb{E}^4_{1}$. If they have the same Gauss map, then their parametrizations are given by 
\begin{equation}
\label{type1izaGH}
   X_{1}(u,v) =(x(u)\cos v,x(u)\sin v,c_1,w(u)+\lambda v)
\end{equation}
and
\begin{equation}
\label{izodönel1}
  R_{1}(u,v)=\begin{bmatrix}
\sqrt{x^2(u)-\lambda^2}\cos{\left(v-\int{\frac{\lambda w'(u)}{x^2(u)-\lambda^2}du}\right)}\\ \sqrt{x^2(u)-\lambda^2}\sin{\left(v-\int{\frac{\lambda w'(u)}{x^2(u)-\lambda^2}du}\right)}\\
c_2\\
\pm\frac{1}{\sqrt{c_3}}\arcsinh{\sqrt{c_3(x^2(u)-\lambda^2)}}+c_4
\end{bmatrix},
\end{equation}
where $c_1, c_2$ and $c_4$ are arbitrary constants, $0<c_3<\frac{1}{\lambda^2}$ and  
\begin{equation}
\label{eqcor1}
w(u)=\pm\Bigg(\sqrt{\frac{1-c_3\lambda^2}{c_3}}
\arcsinh\left({\sqrt{c_3(x^2(u)-\lambda^2)}}\right)-{\lambda}\arctan{\left(
\sqrt{\frac{(1-c_3\lambda^2)(x^2(u)-\lambda^2)}{\lambda^2(1+c_3(x^2(u)-\lambda^2))}}\right)}\Bigg).
\end{equation}
\end{theorem}

\begin{proof}
Assume that $X_{1}$ is a spacelike helicoidal surfaces of type I given by \eqref{eq9} and $R_{1}$ is a spacelike rotational surface of type I  given by \eqref{type1izR} having same Gauss map in $\mathbb{E}^4_{1}$.
From Theorem \ref{type1thm2}, 
we know that they are hyperplanar and minimal. 
Then, we get that $z(u)=c_1$ and $a(u)=0$. 
Thus, the equation \eqref{type1eq7} gives
\begin{equation}
\label{coreq1}
    b^2(u)=\frac{x^2(u)w'^2(u)+\lambda^2x'^2(u)}{x^2(u)x'^2(u)}.
\end{equation}
Since $R_1$ is minimal, from the equation \eqref{type1izm1} 
we have the following differential equation
\begin{equation}
\label{coreq2}
 (x^2(u)-\lambda^2)b'(u)+x(u)x'(u)b(u)=x(u)x'(u)b^3(u)  
\end{equation}
which is a Bernoulli equation. 
Then, the general solution of this equation is found as 
\begin{equation}
\label{coreq3}
   b^2(u)=\frac{1}{1+c_3(x^2(u)-\lambda^2)}   
\end{equation}
for an arbitrary positive constant $c_3$.
Comparing the equations \eqref{coreq1} and \eqref{coreq3}, we get
\begin{equation}
 w(u)
 =\pm\sqrt{1-c_3\lambda^2}
 \int{\frac{x'(u)}{x(u)}\sqrt{\frac{x^2(u)-\lambda^2}{1+c_3(x^2(u)-\lambda^2)}}}du 
\end{equation}
whose solution is given by \eqref{eqcor1} for $0<c_3<\frac{1}{\lambda^2}$. 
Moreover, using the last component of $R_1(u,v)$ in \eqref{type1izR}, 
we have
\begin{equation}
\int\frac{x(u)x'(u)}{\sqrt{(x^2(u)-\lambda^2)(1+c_3(x^2(u)-\lambda^2))}}du=\pm\frac{1}{\sqrt{c_3}}\arcsinh{\sqrt{c_3(x^2(u)-\lambda^2)}}+c_4
\end{equation}
for any arbitrary constant $c_4$.
\end{proof}
We remark that by taking $x(u)=u$ in Theorem \ref{type1thm2a}, the rotational surface given by \eqref{izodönel1} turned into that is given in \cite{Ikawa2}.

In Theorem \ref{type1thm2a}, 
if we take $c_3=\frac{1}{\lambda^2}$, then we obtain $w(u)=0$.
Thus, we give the following corollary.

\begin{corollary}
\label{type1cor1}
A spacelike right helicoidal surface of type I 
given by
\begin{equation}
\label{type1izaGHR}
   X_{1}(u,v) =(x(u)\cos v,x(u)\sin v,c_1,\lambda v)
\end{equation}
is isometric to a spacelike rotational surface given by
\begin{equation}
\label{izodönel1}
  R_{1}(u,v)=\begin{bmatrix}
\sqrt{x^2(u)-\lambda^2}\cos{v}\\ \sqrt{x^2(u)-\lambda^2}\sin{v}\\
c_2\\
\pm\arcsinh{\sqrt{x^2(u)-\lambda^2}}+c_4
\end{bmatrix}
\end{equation}
so that helices on the spacelike right helicoidal surface of type I correspond to parallel circles on the spacelike rotational surface
where $c_1, c_2$ and $c_4$ are arbitrary constants. 
Moreover, the mean curvatures of two surfaces are zero and the Gauss maps of them are same. 
\end{corollary}

Now, we give an example by using Theorem \ref{type1thm2a}. 

\begin{example}
We consider the spacelike profile curve $\beta_{1}$ as follows:
\[\beta_{1} (u)=\left(u, 0, c_1, w(u)\right).\]
If $\lambda=1$, $c_3 =\frac{1}{2}$ $c_4=0$, then isometric surfaces in \eqref{type1izaGH} and \eqref{izodönel1} are given as follows
\[X_{1}(u,v)=\left(u\cos v, u\sin v,c_1,\pm\left(\arcsinh{\left(\sqrt{\frac{u^2-1}{2}}\right)
}-\arctan{\left(\sqrt{\frac{u^2-1}{u^2+1}}\right)}\right)+v\right)\]
and
\begin{equation*}
\label{type1izaGR}
R_{1}(u,v)=\begin{bmatrix}
\sqrt{u^2-1}\cos{\left(v-\frac{1}{u\sqrt{u^4-1}}\right)}\\ \sqrt{u^2-1}\sin{\left(v-\frac{1}{u\sqrt{u^4-1}}\right)}\\
c_2\\
\pm\sqrt{2}\arcsinh{\left(\frac{u^2-1}{\sqrt{2}}\right)}
\end{bmatrix}.    
\end{equation*}
For $1.1\leq u\leq\pi$ and $0\leq v<2\pi$,  
the graphs of spacelike helicoidal surface $X_{1}$ and spacelike rotational surface $R_{1}$ in $\mathbb{E}^3_1$ can be plotted by using Mathematica 10.4 shown in Figure \ref{fig1}.
\begin{figure}[htbp]
\centering
\includegraphics[height=60mm]{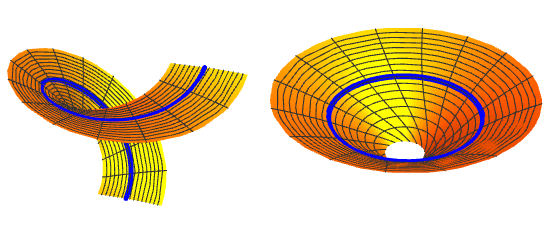}
\caption {\textbf{Left:} Spacelike helicoidal surface of type I; helix, \textbf{Right:} Spacelike rotational surface of type I; circle.}
\label{fig1}
\end{figure}
\end{example}

\subsection{Helicoidal Surface of Type II}
Assume that $X_2$ is a spacelike helicoidal surface of type II in $\mathbb{E}^4_1$ given by \eqref{eq10}. 
From a simple calculation, we have 
\begin{equation}
\label{type2eq2}
g_{11}= x'^2(u)+y'^2(u)-w'^2(u),\;\;g_{12}=g_{21}=-\lambda x'(u),\;\;
g_{22}=w^2(u)+\lambda^2 
\end{equation}
with ${W}=(w^2(u)+\lambda^2)(y'^2(u)-w'^2(u))+x'^2(u)w^2(u)>0$
for all $u\in I\subset\mathbb{R}$.
Without loss of generality, we assume $w'^2(u)-y'^2(u)>0$ for all 
$u\in I\subset\mathbb{R}$. 
Then, we choose an orthonormal frame field 
$\{e_{1},e_{2}, N_{1}, N_{2}\}$ on $X_2$ such that $e_1, e_2$ are tangent to $X_2$ and $N_1, N_2$ are normal to $X_2$:
\begin{align}
\label{type2eq3}
\begin{split}
e_{1}&=\frac{1}{\sqrt{g_{11}}}X_{2u},\;\;\;\;
e_{2}=\frac{1}{ \sqrt{Wg_{11}}}(g_{11}X_{2v}-{g_{12}X_{2u}}), \\
N_{1}&=\frac{1}{\sqrt{w'^2-y'^2}}(0, w', y'\sinh v, y'\cosh v),\\
N_{2}=&\frac{1}{\sqrt{W(w'^2-y'^2)}}(w(w'^2-y'^2),x'y'w, x'ww'\sinh{v}-\lambda(w'^2-y'^2)\cosh{v},\\
&x'ww'\cosh{v}-\lambda(w'^2-y'^2)\sinh{v})
\end{split}
\end{align}
with $\langle e_{1},e_{1}\rangle=\langle e_{2},e_{2}\rangle=\langle N_{1},N_{1}\rangle=-\langle N_{2},N_{2}\rangle=1$. 
By a direct computation, we get the coefficients of the second fundamental form as follows
\begin{align}
\label{type2eq4}
    \begin{split}
     &b_{11}^1=\frac{y'' w'- y' w''}{\sqrt{w'^2-y'^2}},\;\;\; 
     b_{12}^1=b_{21}^1=0,\;\;\;
     b_{22}^1=\frac{-w y'}{\sqrt{w'^2-y'^2}},\\
     &b_{11}^2=\frac{w(x'(y'y''-w'w'')+x''(w'^2-y'^2))}
    {\sqrt{W(w'^2-y'^2)}},\;\; 
     b_{12}^2=b_{21}^2=\frac{-\lambda w'\sqrt{w'^2-y'^2}}{\sqrt{W}},\\
     &b_{22}^2=\frac{- x'w^2w'}{\sqrt{W(w'^2-y'^2)}}.
    \end{split}
\end{align}
Thus, we find the components of the mean curvature $H_{i}^{X_{2}},\;(i=1,2)$ 
and from the equation \eqref{eq6}, we compute the Gauss curvature $K^{X_{2}}$ of $X_2$ in $\mathbb{E}^4_1$, respectively, as follows
\begin{align}
\label{type2eq5} 
&H_{1}^{X_{2}}=\frac{-wy'(x'^2+y'^2-w'^2)+(w^2+
\lambda^2)(y''w'-y'w'')}{2W\sqrt{w'^2-y'^2}},\\
&H_{2}^{X_{2}}=\frac{x'w'(2\lambda^2+w^2)(w'^2-y'^2)-w^2w'x'^3+w(\lambda^2+w^2)(x''(w'^2-y'^2)+x'(y'y''-w'w''))}
{2{W}^{3/2}\sqrt{w'^2-y'^2}}
\end{align}
and 
\begin{equation}
\label{type2eq6}
K^{X_{2}}=\frac{w^3(x'(w'x''-w''x')+y'(w'y''-w''y'))+ \lambda^2(w y'(w'y''-w''y')+w'^2(w'^2-y'^2))}{W^2}.
\end{equation}

The following theorem is a generalization of the classical Bour's theorem  for helicoidal surfaces of type II in $\mathbb{E}^4_{1}$.
\begin{theorem}
\label{type2thm1}
A spacelike helicoidal surface of type II in $\mathbb{E}^4_{1}$ given by \eqref{eq10} is isometric to a spacelike rotational surface 
\begin{equation}
\label{type2izR}
  R_{2}(u,v)=\begin{bmatrix}
\int{\frac{a(u)w(u)w'(u)}{\sqrt{\lambda^2+w^2(u)}}du}\\
\int{\frac{b(u)w(u)w'(u)}{\sqrt{\lambda^2+w^2(u)}}du}\\
\sqrt{\lambda^2+w^2(u)}\sinh{\left(v+\int{\frac{\lambda x'(u)}{\lambda^2+w^2(u)}du}\right)}\\ \sqrt{\lambda^2+w^2(u)}\cosh{\left(v+\int{\frac{\lambda x'(u)}{\lambda^2+w^2(u)}du}\right)}
\end{bmatrix}
\end{equation}
so that helices on the spacelike helicoidal surface of type II correspond to parallel ellipse on the spacelike rotational surface, where $a(u)$ and $b(u)$ are differentiable functions satisfying the following equation:
 \begin{equation}
 \label{type2eq7}
    a^2(u)+b^2(u)
    =\frac{w^2(u)(x'^2(u)+y'^2(u))+\lambda^2(y'^2(u)-w'^2(u))}{w^2(u)w'^2(u)}
 \end{equation}
with $w'(u)\not=0$ for all $u\in I\subset\mathbb{R}$.
\end{theorem}
 
\begin{proof}
Let $X_2$ be a spacelike helicoidal surface in $\mathbb{E}^4_1$ 
defined by \eqref{eq10}. 
From the equation \eqref{type2eq2}, we have the induced metric of $X_2$
as follows:
\begin{equation}
\label{type2mtrc1}
 ds^2 _{X_2}=(x'^2(u)+y'^2(u)-w'^2(u))du^2-2\lambda x'(u)dudv + (\lambda^2+w^2(u))dv^2.
\end{equation}
Then, we find a curve on $X_2$ which is orthogonal to a helix on $X_2$, that is, this curve is $v$--parameter curve of $X_2$. 
From orthogonality condition, we get 
\begin{equation}
\label{type2eq8}
    \lambda x'(u)du+(\lambda^2+w^2(u))dv=0.
\end{equation}
By solving the equation \eqref{type2eq8}, we find 
\begin{equation}
    v=c-\int{\frac{\lambda x'(u)}{\lambda^2+w^2(u)}du},
\end{equation}
where c is an arbitrary constant.
If we take $\bar{v}=c$, then we get
\begin{equation}
\label{type2eq8-a}
    \overline{v}=v+\int{\frac{\lambda x'(u)}{\lambda^2+w^2(u)}du}.
\end{equation}
From equation \eqref{type2eq8-a}, we can easily obtain the following equation:
\begin{equation}
\label{type2eq9}
    dv=d\overline{v}-\frac{\lambda x'(u)}{\lambda^2+w^2(u)}du.
\end{equation}
Substituting the equation \eqref{type2eq9} in \eqref{type2mtrc1}, 
the induced metric of $X_2$ becomes
\begin{equation}
\label{type2eq10}
 ds^2 _{X_2}=\left(x'^2(u)+y'^2(u)-w'^2(u)-\frac{\lambda^2 x'^2(u)}{\lambda^2+w^2(u)}\right)du^2 + (\lambda^2+w^2(u))d\overline{v}^2.
\end{equation}
On the other hand, 
the spacelike rotational surface $R_2$ in $\mathbb{E}^4_{1}$
related to $X_2$ is given by 
\begin{equation}
\label{type2rtk}
    R_{2}(k,t)=(n(k),s(k),r(k)\sinh{t},r(k)\cosh{t}).
\end{equation}
We know that the induced metric of $R_2$ is given by
\begin{equation}
\label{type2eq11}
 ds^2 _{R_2}=(n'^2(k)+s'^2(k)-r'^2(k))dk^2 + r^2(k)dt^2
\end{equation}
with $n'^2(k)+s'^2(k)-r'^2(k)>0$. 
Comparing the equations \eqref{type2eq10} and \eqref{type2eq11}, 
we get an isometry between $X_2$ and $R_2$ by taking
$u=k, \bar{v}=t$, $r(k)=\sqrt{\lambda^2+w^2(u)}$ 
and 
\begin{equation}
\label{type2eq12}
 n'^2(k)+s'^2(k)-r'^2(k)= x'^2(u)+y'^2(u)-w'^2(u)-\frac{\lambda^2 x'^2(u)}{\lambda^2+w^2(u)}. 
\end{equation}
Also, we say $a(u)=\frac{n'(u)}{r'(u)}$ and 
$b(u)=\frac{s'(u)}{r'(u)}$. Then, we find 
\begin{equation}
  n(u)=\int{\frac{a(u)w(u)w'(u)}
  {\sqrt{\lambda^2+w^2(u)}}du}\;\;\mbox{and}\;\; s(u)=\int{\frac{b(u)w(u)w'(u)}
  {\sqrt{\lambda^2+w^2(u)}}du}.
\end{equation}
Thus, we get the spacelike rotational surface $R_2$ given by \eqref{type2izR} which is isometric to $X_2$. 
If we choose a helix as $X_2(u_0,v)$, 
where $u_{0}$ is an arbitrary constant, 
it corresponds to the parallel ellipse $R_{2}(u_0,v)=(0,0,\sqrt{\lambda^2+w^2(u_0)}\sinh{v},\sqrt{\lambda^2+w^2(u_0)}\cosh{v})$ lying on the $x_3x_4$--plane (see \cite{GKG} for details). 
\end{proof} 

For later use, we find the components of the mean curvature vector
of the spacelike rotational surface $R_2$ given by \eqref{type2izR} as follows:
\begin{align}
\label{type2izm1}
\begin{split}
H_1^{R_2}&=\frac{b'(\lambda^2+w^2)+bww'(-1+a^2+b^2))}
{2\sqrt{(1-b^2)(w^2+\lambda^2)(-1+a^2+b^2)}},\\
H_2^{R_2}&=
\frac{-ww'(\lambda^2+w^2)(a'(1+b^2)+abb')-aw'^2(2\lambda^2b^2+w^2(a^2+b^2-1))-2ab^2w''(w^2+\lambda^2)}{2ww'\sqrt{(1-b^2)(a^2+b^2-1)(\lambda^2+w^2)}}.
\end{split}
\end{align}
Then, we consider isometric surfaces according to Bour's theorem whose Gauss maps are same.

\begin{theorem}
\label{type2thm2}
Let $X_{2}$ be a spacelike helicoidal surface of type II given by \eqref{eq10} and $R_{2}$ be a spacelike rotational surface  of type II given by \eqref{type2izR} in $\mathbb{E}^4_{1}$. If they have the same Gauss map, then they are hyperplanar and minimal.
\end{theorem}

\begin{proof}
Let $\{\eta_{1},\eta_{2},\eta_{3},\eta_{4}\}$ be the standard 
orthonormal bases of $\mathbb{E}_1^4$.
Denote $\eta_{ij}=\eta_{i} \wedge \eta_{j}$ for ${i,j=1,2,3,4}$ and $i<j$.
Then, using equation \eqref{eq8}, 
we get the Gauss maps of
$X_2$ and $R_2$, respectively, as follows
\begin{align}
\label{type2GH}
\nu_{X_{2}}=&\frac{1}{\sqrt{W}}(-\lambda y'\eta_{12}+(x'w\cosh{v}-\lambda w'\sinh{v})\eta_{13}+( x'w\sinh{v}-\lambda w'\cosh{v})\eta_{14}\notag\\&+y'w\cosh{v}\eta_{23}   
+y'w\sinh{v}\eta_{24}-ww'\eta_{34})
\end{align}
and
\begin{align}
\label{type2GR}
\begin{split}
\nu_{R_{2}}=&\frac{1}{\sqrt{W}}
\Bigg(aww'\cosh\left(v+\int\frac{\lambda x'(u)}{\lambda^2+w^2(u)}du\right)\eta_{13}+aww'\sinh\left(v+\int\frac{\lambda x'(u)}{\lambda^2+w^2(u)}du\right)\eta_{14}
\\
&
+bww'\cosh\left(v+\int\frac{\lambda x'(u)}{\lambda^2+w^2(u)}du\right)\eta_{23} +bww'\sinh\left(v+\int\frac{\lambda x'(u)}{\lambda^2+w^2(u)}du\right)\eta_{24}
\\
&
-ww'\eta_{34}
\Bigg).
\end{split}
\end{align}
If $X_2$ and $R_2$ have the same Gauss map, 
then comparing \eqref{type2GH} and \eqref{type2GR}, we get the following system of equations
\begin{equation}
  \label{type2eq12}
    \lambda y'(u)=0,
\end{equation}
\begin{equation}
  \label{type2eq13}
    x'(u)w(u)\cosh{v}-\lambda w'(u)\sinh{v}=a(u)
    w(u)w'(u)\cosh\left(v+\int\frac{\lambda x'(u)}{\lambda^2+w^2(u)}du\right),
\end{equation}
\begin{equation}
  \label{type2eq14}
    x'(u)w(u)\sinh{v}-\lambda w'(u)\cosh{v}=a(u)
    w(u)w'(u)\sinh\left(v+\int\frac{\lambda x'(u)}{\lambda^2+w^2(u)}du\right),
\end{equation}
\begin{equation}
  \label{type2eq15}
    y'(u)w(u)\cosh{v}=b(u)w(u)w'(u)\cosh\left(v+\int\frac{\lambda x'(u)}{\lambda^2+w^2(u)}du\right),
\end{equation}
\begin{equation}
  \label{type2eq16}
     y'(u)w(u)\sinh{v}=b(u)w(u)w'(u)\sinh\left(v+\int\frac{\lambda x'(u)}{\lambda^2+w^2(u)}du\right).
\end{equation}
Since $\lambda \neq 0 $, $y'(u)=0$ from equation \eqref{type2eq12}.
Then, the equations \eqref{type2eq15} and \eqref{type2eq16}
implies $b(u)=0$. 
Therefore, it can be easily seen that 
the surfaces $X_{2}$ and $R_{2}$ are hyperplanar,
that is, they are lying in $\mathbb{E}^3_1$. 
In this case, the equation \eqref{type2eq5} and \eqref{type2izm1} gives the components of the mean curvatures of $X_2$ and $R_2$ 
\begin{equation}
\label{type2eq17}
H_{1}^{X_{2}}=H_{1}^{R_{2}}=0
\end{equation}
and 
\begin{align}
\label{type2eq18}
\begin{split}
H_{2}^{X_{2}}&=\frac{x'w'^2(2\lambda^2+w^2)
-w^2x'^3+w(\lambda^2+w^2)(x''w'-x'w'')}{2{W}^{3/2}},\\
H_{2}^{R_{2}}&=
\frac{\sqrt{w^2+\lambda^2}(w^2+\lambda^2)a'+aww'(a^2-1)}
{2W\sqrt{a^2-1}}.
\end{split}
\end{align}
On the other hand, the equation \eqref{type2eq7} gives 
\begin{equation}
 \label{type2eq18a}   
 a^2(u)=\frac{w^2(u)x'^2(u)-\lambda^2 w'^2(u)}{w^2(u)w'^2(u)}. 
\end{equation}
Using the equation \eqref{type2eq18a} in \eqref{type2eq18}, we get
\begin{equation}
 \label{type2eq18b}
H_{2}^{R_{2}}=\frac{w^2x'\sqrt{w^2+\lambda^2}(x'w'^2(2\lambda^2+w^2)
-w^2x'^3+w(\lambda^2+w^2)(x''w'-x'w''))}{2W^{3/2}\sqrt{w^2x'^2-w'^2\lambda^2}}.    
\end{equation}
Thus, it can be seen that $H_2^{R_2}=\frac{w^2x'\sqrt{w^2+\lambda^2}}{W^{1/2}\sqrt{(w^2x'^2-w'^2\lambda^2)(w^2x'^2-w'^2(\lambda^2+w^2))}}H_2^{X_2}$. 
Moreover, using equations \eqref{type2eq14} and \eqref{type2eq15}, we obtain the following equations
\begin{equation}
\label{type2eq21}
    x'(u)w(u)=a(u)w(u)w'(u)\cosh{\left(\int\frac{\lambda x'(u)}{\lambda^2+w^2}du\right)},
\end{equation}
\begin{equation}
\label{type2eq22}
    \lambda w'(u)=-a(u)w(u)w'(u)\sinh{\left(\int\frac{\lambda x'(u)}{\lambda^2+w^2}du\right)}.
\end{equation}
Considering the equations \eqref{type2eq21} and \eqref{type2eq22} together, we have
\begin{equation}
\label{type2eq23}
    \frac{-x'(u)w(u)}
    {\lambda w'(u)}=\coth{\left(\int\frac{\lambda x'(u)}{\lambda^2+w^2}du\right)}.
\end{equation}
If we take the derivative of the equation \eqref{type2eq23} with respect to $u$, \eqref{type2eq23} becomes
\begin{equation}
\label{type2eq24}
    \lambda(x'w'^2(2\lambda^2+w^2)
-w^2x'^3+w(\lambda^2+w^2)(x''w'-x'w''))=0
\end{equation}
which implies $H_{2}^{X_{2}}=H_{2}^{R_{2}}=0$ in the equation 
\eqref{type2eq18}. 
Thus, we get the desired results.
\end{proof}

\begin{theorem}
\label{type2cor1}
Let $X_{2}$ be a spacelike helicoidal surfaces of type II given by \eqref{eq10} and $R_{2}$ be a spacelike rotational surface given by \eqref{type2izR} in $\mathbb{E}^4_{1}$. If they  have  the  same  Gauss map, then their parametrizations are given by
\begin{equation}
\label{type2izaGH}
   X_{2}(u,v) =\left(x(u)+\lambda v,c_1,w(u)\sinh v,w(u)\cosh v\right)
\end{equation}
and
\begin{equation}
\label{type2izaGR}
  R_{2}(u,v)=\begin{bmatrix}
\pm\frac{1}{\sqrt{-c_3}}\arcsin{\sqrt{-c_3\left(\lambda^2+w^2(u)\right)}}\\
c_2\\
\sqrt{\lambda^2+w^2(u)}\sinh{\left(v-\int{\frac{\lambda x'(u)}{\lambda^2+w^2(u)}du}\right)}\\ \sqrt{\lambda^2+w^2(u)}\cosh{\left(v-\int{\frac{\lambda x'(u)}{\lambda^2+w^2(u)}du}\right)}\\
\end{bmatrix},
\end{equation}
where $c_1, c_2$ and $c_4$ are arbitrary constants, 
$-\frac{1}{\lambda^2}<c_3<0$ and  
\begin{equation}
\label{eqcor2}
\begin{split}
x(u)=&\pm\frac{1}{\sqrt{-c_3}}\Bigg(
\arcsin{\sqrt{-c_3(\lambda^2+w^2(u)}}
-\frac{\lambda\sqrt{-c_3}}{\sqrt{1+c_3\lambda^2}}
\arctan{\left(\frac{\sqrt{(1+c_3\lambda^2)(\lambda^2+w^2(u))}}
{\lambda \sqrt{1+c_3(\lambda^2+w^2(u))}}\right)}\Bigg).    
\end{split}
\end{equation}
\end{theorem}

\begin{proof}
Assume that $X_{2}$ is a spacelike helicoidal 
surface of type II in $\mathbb{E}^4_{1}$ given by \eqref{eq10} and $R_{2}$ is a spacelike rotational surface of type II in $\mathbb{E}^4_{1}$ given by \eqref{type2izR} having same Gauss map. 
From Theorem \ref{type2thm2}, 
we know that they are hyperplanar and minimal. 
Then, we get that $y(u)=c_1$ and $b(u)=0$. 
Thus, the equation \eqref{type2eq7} gives
\begin{equation}
a^2(u)=\frac{x'^2(u)w^2(u)-\lambda^2 w'^2(u)}{w^2(u)w'^2(u)}.
\end{equation}
Since $R_2$ is minimal, from the equation \eqref{type2eq18} 
we have the following differential equation
\begin{equation}
\label{coreqtype2}
 (\lambda^2+w^2(u))a'(u)+w(u)w'(u)a(u)=w(u)w'(u)a^3(u)  
\end{equation}
which is a Bernoulli equation. 
Then, the general solution of this equation is found as 
\begin{equation}
\label{cortype2eq1}
   a^2(u)=\frac{1}{1+c_3(\lambda^2+w^2(u))}   
\end{equation}
for an arbitrary negative constant $c_3$. Comparing the equations \eqref{coreqtype2} and \eqref{cortype2eq1}, we get
\begin{equation}
 x(u)
 =\pm\sqrt{1+c_3\lambda^2}
 \int{\frac{w'(u)}{w(u)}
 \sqrt{\frac{w^2(u)+\lambda^2}{1+c_3(w^2(u)+\lambda^2)}}}du  
\end{equation}
whose solution is given by \eqref{eqcor2} for $\frac{-1}{\lambda^2}<c_3<0$.
Moreover, using the last component of $R_2(u,v)$ in \eqref{type2izR}, 
we have
\begin{equation}
\int\frac{w(u)w'(u)}
{\sqrt{(w^2(u)+\lambda^2)(1+c_3(w^2(u)+\lambda^2))}}du
=\pm\frac{1}{\sqrt{-c_3}}\arcsin{\sqrt{-c_3(w^2(u)+\lambda^2)}}+c_4
\end{equation}
for any arbitrary constant $c_4$.
\end{proof}
By taking $w(u)=u$ in Theorem \ref{type2cor1}, 
the rotational surface given by \eqref{type2izaGR} was obtained in \cite{Ikawa2}.

We note that if $x'(u)=0$, then the helicoidal surface of type II reduces to the right helicoidal surface of type II in $\mathbb{E}^4_1$.
From the equation \eqref{type2izR}, 
we find the parametrization of the rotational surface isometric to the right helicoidal surface.
Under condition that they have same Gauss maps, the causal character of the right helicoidal surface changes which gives a contradiction.
Thus, their Gauss maps are definitely different. 

Now, we give an example by using Theorem \ref{type2cor1}.

\begin{example}
We consider the spacelike profile curve $\beta_{2}$ as follows:
\[\beta_{2} (u)=\left(x(u),c_1, 0, u\right).\]
If $\lambda=1$, $c_3 =\frac{1}{2}$ $c_4=0$, then spacelike helicoidal surface of type II and spacelike rotational surface of type II is given as follows, respectively.
\begin{equation*}
X_{2}(u,v)=\begin{bmatrix}
\sqrt{2}\left(\arcsin{\left(\sqrt{\frac{u^2+1}{2}}\right)}-\arctan{\left(\sqrt{\frac{u^2+1}{1-u^2}}\right)}\right)+v\\ c_1\\u\sinh{v}\\u\cosh{v}
\end{bmatrix}   
\end{equation*}
and
\begin{equation*}
R_{2}(u,v)=\begin{bmatrix}
\sqrt{2}\arcsin{\left(\sqrt\frac{u^2+1}{2}\right)}\\
c_2\\
\sqrt{u^2+1}\sinh v\\ \sqrt{u^2+1}\cosh v
\end{bmatrix}.    
\end{equation*}
For $0\leq u\leq\pi$ and $0\leq v\leq\frac{\pi}{4}$,
the graphs of spacelike helicoidal surface $X_{2}$ and spacelike rotational surface $R_{2}$ in $\mathbb{E}^3_1$ can be plotted by using Mathematica 10.4 shown in Figure \ref{fig2}.
\begin{figure}[htbp]
\centering
\includegraphics[height=60mm]{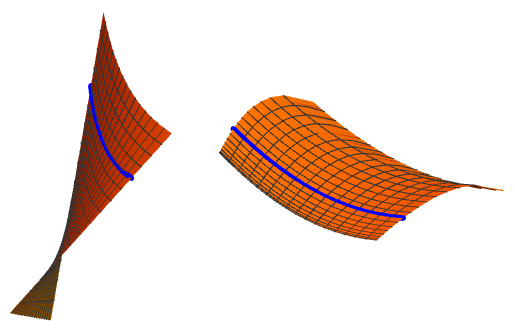}
\caption {\textbf{Left:} Spacelike helicoidal surface of type II; helix, \textbf{Right:} Spacelike rotational surface of type II; ellipse.}
\label{fig2}
\end{figure}
\end{example}

\subsection{Helicoidal Surface of Type III}
Assume that $X_3$ is a spacelike helicoidal surface of type III in $\mathbb{E}^4_1$ given by \eqref{eq11}. 
From a simple calculation, we have 
\begin{equation}
\label{type3eq2}
g_{11}= x'^2(u)-2w'(u)z'(u),\;\;g_{12}=g_{21}=-\lambda w'(u),\;\;
g_{22}=2w^2(u) 
\end{equation}
with ${W}=2w^2(u)(x'^2(u)-2w'(u)z'(u))-\lambda^2w'^2(u)>0$
for all $u\in I\subset\mathbb{R}$.
Then, we choose an orthonormal frame field 
$\{e_{1},e_{2}, N_{1}, N_{2}\}$ on $X_3$ such that $e_1, e_2$ are tangent to $X_3$ and $N_1, N_2$ are normal to $X_3$:
\begin{align}
\label{type3eq3}
\begin{split}
e_{1}&=\frac{1}{\sqrt{g_{11}}}X_{3u},\;\;\;\;
e_{2}=\frac{1}{\sqrt{Wg_{11}}}(g_{11}X_{3v}-{g_{12}X_{3u}}), \\
N_{1}&=\eta_{1}+\frac{x'}{w'}\xi_{3},\\
N_{2}=&\frac{1}{w'\sqrt{W}}\bigg(\sqrt{2}x'ww'\eta_{1}+(\lambda w'^2+2vww'^2)\eta_{2}+\sqrt{2}(\lambda vw'^2+v^2ww'^2+wz'^2-ww'z'^2)\xi_{3}\\
&+\sqrt{2}ww'^2\xi_{4}\bigg)
\end{split}
\end{align}
with $\langle e_{1},e_{1}\rangle=\langle e_{2},e_{2}\rangle=\langle N_{1},N_{1}\rangle=-\langle N_{2},N_{2}\rangle=1$. 
By direct computations, we get the coefficients of the second fundamental form as follows
\begin{align}
\label{type3eq4}
    \begin{split}
     &b_{11}^1=\frac{x'' w'- x' w''}{w'},\;\;\; 
     b_{12}^1=b_{21}^1=b_{22}^1=0,\\
     &b_{11}^2
     =\frac{\sqrt{2}(x'x''w'-x'^2w''+w'(z'w''-w'z''))}{w'\sqrt{W}},\;\; 
     b_{12}^2=b_{21}^2=\frac{\sqrt{2}\lambda w'^2}{\sqrt{W}},\;\;
     b_{22}^2=\frac{-2\sqrt{2}w^2w'}{\sqrt{W}}.
    \end{split}
\end{align}
Thus, we find the components of the mean curvature $H_{i}^{X_{3}},\;(i=1,2)$ 
and from the equation \eqref{eq6}, we find the components of the Gauss curvature $K^{X_{3}}$ of $X_3$ in $\mathbb{E}^4_1$, respectively, as follows
\begin{align}
\label{type3eq5} 
\begin{split}
&H_{1}^{X_{3}}=\frac{w^2(x''w'-x'w'')}
{2w'\sqrt{W}},\\
&H_{2}^{X_{3}}=\frac{\sqrt{2}(-\lambda^2w'^4
-2w^2w'^3z'+w^3x'^2w''-w^3w'(z'w''+x'x'')
+w^2w'^2(x'^2+wz''))}
{w'\sqrt{W}}
\end{split}
\end{align}
and 
\begin{equation}
\label{type1eq6}
K^{X_{3}}=\frac{2\lambda^2w'^4-4w^3(x'^2w''+w'(z'w''+x'x''-w'z''))}{W^2}.
\end{equation}
The following theorem is a generalization of the classical Bour's theorem  for a spacelike helicoidal surfaces of type III in  $\mathbb{E}^4_{1}$.
\begin{theorem}
\label{type3thm1}
A spacelike helicoidal surface of type III in $\mathbb{E}^4_{1}$ given by \eqref{eq11} is isometric to a spacelike rotational surface
\begin{align}
\label{type3izR}
\begin{split}
R_{3}(u,v)&=
\int{a(u)w'(u)du}\eta_{1}+
\sqrt{2}w(u)\left(v-\frac{\lambda}{2w(u)}\right)\eta_{2}\\
&
+\left(\int{b(u)w'(u)du}+w(u)\left(v-\frac{\lambda}{2w(u)}\right)^2\right)\xi_{3}+w(u)\xi_{4}    
\end{split}
\end{align}
so that helices on the spacelike helicoidal surface of type III correspond to straight lines on the spacelike rotational surface, where $a(u)$ and $b(u)$ are differentiable functions satisfying the following equation:
\begin{equation}
 \label{type3eq7}
    a^2(u)-2b(u)=\frac{x'^2(u)-2w'(u)z'(u)}
    {w'^2(u)}-\frac{\lambda^2}{2w^2(u)}
 \end{equation}
with $w'(u)\not=0$ for all $u\in I\subset\mathbb{R}$.
\end{theorem}
\begin{proof}
Let $X_3$ be a spacelike helicoidal surface defined by \eqref{eq11} and $R_3$ be a spacelike rotational surface corresponding to $X_3$ in $\mathbb{E}^4_1$. 
From the equation \eqref{type3eq2}, we have the induced metric of $X_3$
as follows:
\begin{equation}
\label{type3mtrc1}
 ds^2 _{X_3}=(x'^2(u)-2w'(u)z'(u))du^2
 -2\lambda w'(u)dudv + 2w^2(u)dv^2.
\end{equation}
Then, we find a curve on $X_3$ which is orthogonal to a helix on $X_3$, that is $v$--parameter curve of $X_3$. 
From orthogonality condition, we get 
\begin{equation}
\label{type3eq8}
    -\lambda w'(u)du+2w^2(u)dv=0.
\end{equation}
By solving the equation \eqref{type3eq8}, we find 
\begin{equation}
    v=-\frac{\lambda}{2w(u)}+c,
\end{equation}
where c is an arbitrary constant.
If we take $\bar{v}=c$, then we have
\begin{equation}
\label{type3eq8-a}
v=\overline{v}-\frac{\lambda}{2w(u)}.
\end{equation}
From equation \eqref{type3eq8-a}, we can easily obtain the following equation:
\begin{equation}
\label{type3eq9}
    dv=d\overline{v}+\frac{\lambda w'(u)}{2w^2(u)}du.
\end{equation}
Substituting the equation \eqref{type3eq9} in \eqref{type3mtrc1}, 
the induced metric of $X_3$ becomes
\begin{equation}
\label{type3eq10}
 ds^2 _{X_3}=\left(x'^2(u)-2w'(u)z'(u)-\frac{\lambda^2 w'^2(u)}{2w^2(u)}\right)du^2 + (2w^2(u))d\overline{v}^2.
\end{equation}
On the other hand, 
the spacelike rotational surface $R_3$ in $\mathbb{E}^4_{1}$
related to $X_3$ is given by 
\begin{equation}
\label{type3rtk}
R_{3}(k,t)=n(k)\eta_{1}+\sqrt{2}tr(k)\eta_{2}+(s(k)+t^2r(k))\xi_{3}+r(k)\xi_{4}.
\end{equation}
We know that the induced metric of $R_3$ is given by
\begin{equation}
\label{type3eq11}
 ds^2 _{R_3}=(n'^2(k)-2r'(k)s'(k))dk^2 + 2r^2(k)dt^2.
\end{equation}
with $n'^2(k)-2r'(k)s'(k)>0$. 
Comparing the equations \eqref{type3eq10} and \eqref{type3eq11}, 
we get an isometry between $X_3$ and $R_3$ by taking
$u=k, \bar{v}=t$, $r(k)=w(u)$ 
and 
\begin{equation}
\label{type3eq12}
 n'^2(k)-2r'(k)s'(k)= \left(x'^2(u)-2w'(u)z'(u)-\frac{\lambda^2 w'^2(u)}{2w^2(u)}\right). 
\end{equation}
Say $a(u)=\frac{n'(k)}{r'(k)}$ and 
$b(u)=\frac{s'(k)}{r'(k)}$. Then,
the equation \eqref{type3eq12} gives \eqref{type3eq7}. 
Moreover, we find 
\begin{equation}
  s(u)=\int{b(u)w'(u)du}\;\;\mbox{and}\;\; n(u)=\int{b(u)w'(u)du}.
\end{equation}
Thus, we get the spacelike rotational surface $R_3$ given by \eqref{type3izR} which is isometric to $X_3$  given by \eqref{eq11}. 
If we choose a helix as $X_3(u_0,v)$, 
where $u_{0}$ is an arbitrary constant, 
it corresponds to the parallel straight line
\begin{align}
\label{type3izR0}
\begin{split}
R_{3}(u_0,v)&=
\sqrt{2}w(u_0)\left(v-\frac{\lambda}{2w(u_0)}\right)\eta_{2}
+w(u_0)\left(v-\frac{\lambda}{2w(u_0)}\right)^2\xi_{3}+w(u_0)\xi_{4}    
\end{split}
\end{align} (see \cite{Grbovic} for details). 
\end{proof} 

\begin{theorem}
The Gauss maps of a spacelike helicoidal surface of type III in $\mathbb{E}^4_1$ given by \eqref{eq11} and a spacelike rotational surface in $\mathbb{E}^4_1$
given by \eqref{type3izR} are definitely different.
\end{theorem}

\begin{proof}
Let $\{\eta_{1},\eta_{2},\xi_{3},\xi_{4}\}$ be the pseudo orthonormal bases of $\mathbb{E}_1^4$.
Then, using equation \eqref{eq8} 
we get the Gauss maps of
$X_3$ and $R_3$, respectively, as follows
\begin{align}
\label{type3GH}
\begin{split}
\nu_{X_{3}}=&\frac{1}{\sqrt{W}}(\sqrt{2}x'w\eta_1\wedge\eta_2
+x'(\lambda+2vw(u))\eta_1\wedge\xi_3\\
&+ \sqrt{2}(v^2ww'-wz'+\lambda vw')\eta_2\wedge\xi_3
-\sqrt{2}ww'\eta_2\wedge\eta_4-w'(\lambda+2vw)\xi_3\wedge\xi_4)
\end{split}
\end{align}
and
\begin{align}
\label{type3GR}
\begin{split}
\nu_{R_{3}}=&\frac{1}{\sqrt{W}}
(\sqrt{2}aww'\eta_1\wedge\eta_2
+2aww'\left(v-\frac{\lambda}{2w}\right)\eta_1\wedge\xi_3\\
&+\sqrt{2}ww'\left(\left(v-\frac{\lambda}{2w}\right)^2-b\right)\eta_2\wedge\xi_3
-\sqrt{2}ww'\eta_2\wedge\xi_4
-ww'\left(2v-\frac{\lambda}{w}\right)\xi_3\wedge\xi_4).
\end{split}
\end{align}
Comparing \eqref{type3GH} and \eqref{type3GR}, it can be easily seen that they are identical when $\lambda=0$ or $w'(u)=0$. 
That is a contradiction. Thus, their Gauss maps are certainly different.
\end{proof}

We now give an example by using Theorem \ref{type3thm1}. 

\begin{example}
We consider the spacelike profile curve $\beta_{3}$ as follows:
\[\beta_{3} (u)=u\eta_{1}+c\xi_{3}+u\xi_{4}.\]
Then, spacelike helicoidal surface of type III and spacelike rotational surface of type III is given as follows, respectively
\begin{equation}
\label{type3aGH}
   X_{3}(u,v) =u\eta_{1}+\sqrt{2}uv\eta_{2}+\left(c+uv^2+\lambda v\right)\xi_{3}+u\xi_{4}
\end{equation}
and
\begin{equation}
    R_{3}(u,v)=u\eta_{1}+\sqrt{2}uv\eta_{2}+\left(c+uv^2\right)\xi_{3}+u\xi_{4}.
\end{equation}
If we take $c=0$ and $\lambda=1$, then the graphs of a spacelike helicoidal surface $X_{3}$ and a spacelike rotational surface $R_{3}$ in $\mathbb{E}^3_1$ can be plotted by using Mathematica 10.4
with the following command:
\[ParametricPlot3D[\{x(u,v)+y(u,v),z(u,v),w(u,v)\},\{u,a,b\},\{v,c,d\}]\]
with $0\leq u\leq\pi$ and $-\pi\leq v\leq\pi$ shown in Figure \ref{fig3}.

\begin{figure}[htbp]
\centering
\includegraphics[height=60mm]{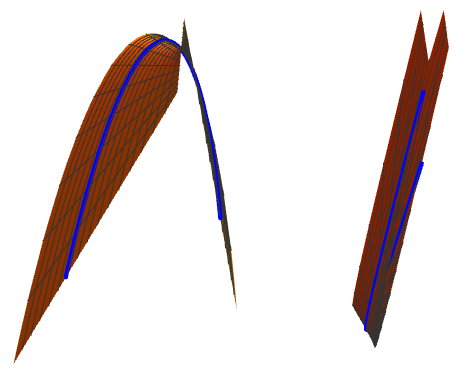}
\caption {\textbf{Left:} Spacelike helicoidal surface of type III ; helix, \textbf{Right:} Spacelike rotational surface of type III; straight line.}
\label{fig3}
\end{figure}
\end{example}

\newpage
\section{Conclusion}

In this paper, we investigate Bour's theorem for three types of spacelike helicoidal surfaces in Minkowski 4-space which are obtained by M. Babaarslan and N. S{\"o}nmez \cite{Babaarslan4}. Moreover, we give the parametrizations of the pair of isometric spacelike surfaces which have same Gauss map and investigate their geometric properties such as hyperplanar and minimal. Finally, we give some examples and plot the corresponding graphs by using Mathematica.

In the future, similarly we will study on Bour's theorem on timelike helicoidal surfaces in Minkowski 4--space. Furthermore, we will investigate the geometric properties of the pair of isometric surfaces according to Bour's theorem in Minkowski 4-space whose Gauss maps are same.

\section*{Acknowledgment}
This work is a part of the master thesis of the first author and it is supported by The Scientific and Technological Research Council of Turkey (TUBITAK) under Project 121F211.


\begin{thebibliography}{}

\bibitem{Babaarslan1} 
Babaarslan, M., Yayl{\i}, Y.:
\textit{Differential Equation of the Loxodrome on a Helicoidal Surface.} Journal of Navigation 2015; 68, 962--970.

\bibitem{Babaarslan2} 
Babaarslan, M., Kayac{\i}k, M.:
\textit{Time-like Loxodromes on Helicoidal Surfaces in Minkowski 3-Space.} Filomat 2017; 31, 4405--4414.

\bibitem{Babaarslan3} 
Babaarslan, M., Kayac{\i}k, M.:
\textit{Differential Equations of the Space-Like Loxodromes on the Helicoidal Surfaces in Minkowski 3-Space.} Differential Equations and Dynamical Systems 2020; 28, 495--512.

\bibitem{Babaarslan4}
Babaarslan, M., S{\"o}nmez, N.: 
\textit{Loxodromes on non--degenerate helicoidal Surfaces in Minkowksi space--time.}
Indian Journal of Pure and Applied Mathematics 2021, https://doi.org/10.1007/s13226-021-00030-x.

\bibitem{Baikoussis} 
Baikoussis, C., Koufogiorgos, T.:
\textit{Helicoidal surfaces with prescribed mean or Gaussian curvature.} Journal of Geometry 1998; 63, 25--29.

\bibitem{Bektas} 
Bekta\c{s}, B., Dursun, U.:
\textit{Timelike Rotational Surface of Elliptic, Hyperbolic and Parabolic Types in Minkowski Space $\mathbb{E}^4_{1}$ with Pointwise 1-Type Gauss Map.} Filomat 2015; 29, 381--392.

\bibitem{Beneki} 
Beneki, C. C., Kaimakamis, G., Papantoniou, B. J.:
\textit{Helicoidal surfaces in three-dimensional Minkowski space.} Journal of Mathematical Analysis and Applications 2002; 275, 586--614.

\bibitem{Bour} 
Bour, E.:
\textit{Memoire sur le deformation de surfaces.} Journal de l'Ecole Polytechnique, XXXIX Cahier 1862; 1--148.

\bibitem{Chen1} 
Chen, B.Y., Li, S. J.: 
\textit{Spherical Hypersurfaces with 2-Type Gauss Map.} Beiträge zur Algebra und Geometrie 1998; 39, 169--179.

\bibitem{Choi1} 
Choi, M. K., Kim, Y. H., Park, G. C.:
\textit{Helicoidal surfaces and their Gauss map in Minkowski 3-space II.} Bulletin of the Korean Mathematical Society 2009; 46, 567--576.

\bibitem{Choi2} 
Choi, M. K., Kim, Y. H., Liu, H., Yoon, D. W.:
\textit{Helicoidal surfaces and their Gauss map in Minkowski 3-space.} Bulletin of the Korean Mathematical Society 2010; 47, 859--881.

\bibitem{Carmo} 
Do Carmo, M. P.,  Dajczer, M.:
\textit{Helicoidal surfaces with constant mean curvature.} Tohoku Mathematical Journal 1982; 34, 425--435.

\bibitem{Dursun} 
Dursun, U., Bekta\c{s}, B.:
\textit{Spacelike Rotational Surface of Elliptic, Hyperbolic and Parabolic Types in Minkowski Space $\mathbb{E}^4_{1}$ with Pointwise 1-Type Gauss Map.} Mathematical Physics Analysis and Geometry 2014; 17, 247--263.

\bibitem{GKG} 
G{\"o}z{\"u}tok, A. \c{C}., Karaku\c{s} , S. {\"O}., G{\"u}ndo\u{g}an, H.:
\textit{Conics and Quadrics in Lorentzian Space.} Mathematical Sciences and Applications E-Notes 2018; 6, 58--63.

\bibitem{Grbovic} 
Grbovi\'{c}, M., Ne\v{s}ovi\'{c}, E.:
\textit{Some relations between rectifying and normal curves in
Minkowski 3-space.} Mathematical Communications 2012; 17, 655--664.

\bibitem{Guler1} 
G{\"u}ler, E., Yayl\i, Y., Hac\i saliho{\u g}lu, H. H.:
\textit{Bour's Theorem on the Gauss Map in 3-Euclidean Space.} Hacettepe Journal of Mathematics and Statistics 2010; 39, 515--525.

\bibitem{Guler2} 
G{\"u}ler, E., Yayl\i, Y.:
\textit{Generalized Bour's Theorem.} Kuwait Journal of Science 2010; 42, 79--90.

\bibitem{Guler3} 
G{\"u}ler, E., Vanl{\i}, A. T.:
\textit{Bour's Theorem in Minkowski 3-space.}  Journal Mathematics of Kyoto University 2006; 46, 47--63.

\bibitem{Hieu} 
Hieu, D. T., Thang, N. N.:
\textit{Bour's Theorem in 4-dimensional Euclidean Space.}  Journal of the Korean Mathematical Society 2017; 54, 2081--2089.

\bibitem{Ikawa1} 
Ikawa, T.: 
\textit{Bour's Theorem and Gauss Map.} Yokohama Mathematical Journal  2000; 48, 173--180.

\bibitem{Ikawa2} 
Ikawa, T.: 
\textit{Bour's Theorem in Minkowski Geometry.} Tokyo Journal Mathematics 2001; 24, 377--394.

\bibitem{Ji1} 
Ji, F., Hou, Z. H.:
\textit{A kind of helicoidal surfaces in 3-dimensional Minkowski space.} Journal of Mathematical Analysis and Applications 2005; 304, 632--643.

\bibitem{Ji2} 
Ji, F., Kim, Y. H.:
\textit{Mean curvatures and Gauss maps of a pair of isometric helicoidal and rotation surfaces in Minkowski 3-space.} Journal of Mathematical Analysis and Applications 2010; 368, 623--635.

\bibitem{Ji3} 
Ji, F., Kim, Y. H.:
\textit{Isometries between minimal helicoidal surfaces and rotation surfaces in Minkowski space.} Applied Mathematics and Computation 2013; 220, 1--11.

\bibitem{Karacan} 
Karacan, M. K., Tuncer Y., Y{\"u}ksel, N.:
\textit{Classifications of conformal rotational surfaces in Euclidean 3-space.} Applied Mathematics E-Notes 2020; 20, 535--544.

\bibitem{Kim} 
Kim, Y. H., Turgay, N. C.:
\textit{Classifications of helicoidal surfaces with  $L_1$-pointwise 1-type Gauss map.} Bulletin of the Korean Mathematical Society 2013; 50, 1345--1356.

\bibitem{Lopez} 
Lopez, R., Demir, E.:
\textit{Helicoidal surfaces in Minkowski space with constant mean curvature and constant Gauss curvature.} Open Mathematics 2014; 12, 1349--1361.

\bibitem{Oprea} 
Oprea, J.: 
\textit{Differential Geometry and its Applications.} New Jersey: MMA Press, 1997.

\bibitem{Ratcliffe} 
Ratcliffe, J. G.: 
\textit{Foundations of hyperbolic manifolds.} Springer Graduate texts in mathematics, 149, second edition, 2006.

\bibitem{ChenP}
Chen, B.-Y and Piccinni, P.: \textit{Submanifolds with finite type Gauss map.} Bulletin of the Australian Mathematical Society 1987; 35(2), 161--186.
\end{thebibliography}
\end{document}